\documentclass[a4paper,12pt,reqno]{amsart}     
\usepackage{amssymb,amsthm,amsmath} 
\usepackage[english]{babel} 
\usepackage{fullpage}    
\usepackage{hyperref}  
\numberwithin{equation}{section}
\newtheorem{theo}{Theorem}[section]
\newtheorem{lemm}[theo]{Lemma}
\newtheorem{prop}[theo]{Proposition}
\newtheorem{defi}[theo]{Definition}

\newcommand{\ii}{{\mathrm{i} }}

\newcommand{\asf}{\,{\sf a}}
\newcommand{\fsf}{\,{\sf f}}

\newcommand{\nn}{\nonumber}
\newcommand{\ce}{\mathcal{E}}
\newcommand{\dd}{\mathrm{d}}
\newcommand{\ca}{\mathcal{A}}

\newcommand{\cl}{\mathcal{L}}

\newcommand{\cb}{\mathcal{B}}

\newcommand{\pq}{\IC\mathrm{P}^1_{q}}  
\newcommand{\pqn}{\IC\mathrm{P}^n_{q}}  
\newcommand{\Apq}{{\mathcal O}(\IC\mathrm{P}^1_{q})}  
\newcommand{\cu}{\mathcal{U}}        
\newcommand{\SU}{\mathrm{SU}_q(2)}  
\newcommand{\ASU}{{\mathcal O}(\mathrm{SU}_q(2))}  
\newcommand{\sq}{\mathrm{S}^2_{q}}  
\newcommand{\Asq}{\ca(\mathrm{S}^2_{q})}  
\newcommand{\su}{\cu_q(\mathrm{su}(2))}  

\newcommand{\cop}{\Delta}           
\newcommand{\co}[2]{#1_{(#2)}}      
\newcommand{\hs}[2]{\left\langle #1,#2\right\rangle}  

\newcommand{\ket}[1]{\left | #1 \right\rangle }

\newcommand{\lt}{{\triangleright}}    
\newcommand{\rt}{{\triangleleft}}
\newcommand{\IC}{{\mathbb C}} 
\newcommand{\IR}{{\mathbb R}} 
\newcommand{\IN}{{\mathbb N}} 
\newcommand{\IT}{{\mathbb T}} 
\newcommand{\IZ}{{\mathbb Z}} 
\DeclareMathOperator{\Aut}{Aut}
\DeclareMathOperator{\Hom}{Hom}
\DeclareMathOperator{\id}{id}       
\DeclareMathOperator{\Mat}{Mat}       
\DeclareMathOperator{\U}{U}       
\DeclareMathOperator{\tr}{tr}       
\DeclareMathOperator{\qtr}{tr_q}       

\newcommand{\beqa}{\begin{align}}
\newcommand{\eeqa}{\end{align}}
\newcommand{\beq}{\begin{equation}}
\newcommand{\eeq}{\end{equation}}

\newcommand{\dol}{\partial}
\newcommand{\dolb}{\overline{\dol}}

\newcommand{\abs}[1]{\left|#1\right|}
\newcommand{\mn}{\abs{n}}
\newcommand{\qp}{{\sf{P}}}
\newcommand{\qpp}{{\sf{P}}^{\Psi}}
\newcommand{\qpn}{{\sf{P}}^{\Phi}}
\newcommand{\bz}{B_{0}}
\newcommand{\bp}{B_{+}}
\newcommand{\bm}{B_{-}}

\begin{document}


\title{Twisted sigma-model solitons on \\ ~ \\ the quantum projective line}

\author{Giovanni Landi}
\address{Universit\`{a} di Trieste, Trieste, Italy 
and I.N.F.N. Sezione di Trieste, Trieste, Italy. 
email: landi@units.it}

\date{13 July 2016}

\begin{abstract}
On the configuration space of projections in a noncommutative algebra, and for an automorphism 
of the algebra, we use a twisted Hochschild cocycle 
for an action functional, and a twisted cyclic cocycle 
for a topological term. The latter is  
Hochschild-cohomologous to the former and positivity in twisted Hochschild cohomology results into a lower bound for the action functional. While the equations for the critical points are rather complicate, the use of the positivity and the bound by the topological term leads to self-duality equations (thus yielding twisted noncommutative sigma-model solitons, or instantons). We present explicit non-trivial solutions on the quantum projective line.

\bigskip\bigskip\bigskip

\centerline{\emph{A Beppe con affetto e gratitudine per il suo settantesimo compleanno}}

\end{abstract}

\maketitle

\tableofcontents

\parskip = 1.2 ex

\section{An introduction by way of a brief recalling}
Noncommutative analogues of non-linear sigma-models --- or noncommutative harmonic maps, or even minimally embedded surfaces --- were introduced in \cite{DKL00,DKL03}. 
With a noncommutative {\em source} space one could also consider a {\em target} space made of two points. While the commutative theory would be trivial, the noncommutative counterpart 
becomes a theory of projections in (matrices with entries in) the source algebra, and shows remarkable and far from trivial properties. In particular a natural action functional led to self-duality equations for projections in the source algebra, with solutions having non-trivial topological content thus leading to noncommutative solitons. 

An ordinary non-linear sigma-model is a field theory whose 
configuration space consists of maps $X$ from
the source space, a Riemannian manifold $(\Sigma, g)$  which 
we take to be compact and
orientable, to a target space, an other Riemannian manifold $(M, G)$.
One has an action functional --- the energy of the maps --- given in local coordinates by
$$
S[X] = \frac{1}{2\pi} \int_{\Sigma}
\sqrt{g} ~g^{\mu\nu}\,G_{ij}(X)\, \partial_{\mu}X^{i}\,\partial_{\nu}X^{j} 
\, \dd x ,
$$
where $g=\det(g_{\mu\nu})$ and $g^{\mu\nu}$ is the 
inverse of $g_{\mu\nu}$; moreover $\mu, \nu = 1, \dots, \dim\Sigma$, and $i, j = 1, \dots, 
\dim M$. 
The stationary points of this action functional 
are harmonic 
maps from $\Sigma$ to
$M$ and describe minimal surfaces embedded in $M$.

When $\Sigma$ is two dimensional, the action functional $S[X]$ above is conformally 
invariant, that is to say 
it is left invariant by any
rescaling of the metric $g \to g \, e^\phi$, with $\phi$ 
any map
from $\Sigma$ to $\IR$. As a consequence, the action only depends
on the conformal class of the metric and may be rewritten using a 
complex structure
on $\Sigma$ as
$$
S[X] =  \frac{1}{\pi}\int_{\Sigma}\,G_{jk}(X)\,\partial
X^{j}\wedge\overline{\partial}X^{k} ,
$$
where $\partial=\dd z \, \partial_{z} $ and
$\overline{\partial}=\dd \overline{z} \, \partial_{\overline{z}}$, with $z$  
a suitable local complex
coordinate.

The noncommutative generalization proposed in \cite{DKL00,DKL03} started with the $*$-algebras $\ca$ and $\cb$ of complex valued smooth functions defined 
respectively on $\Sigma$ and $M$. Embeddings $X$ of $\Sigma$ into $M$ are then in one to one correspondence with $*$-algebra morphisms $\pi$ from $\cb$ to $\ca$, the correspondence 
being simply given by pullback, $\pi_{X}(f)=f\circ X$.

All of this makes sense for fixed not necessarily commutative
algebras $\ca$ and $\cb$ (over $\IC$ and for simplicity taken to be unital).
Then, the configuration space consists of all $*$-algebra morphisms
from $\cb$ (the target algebra) to $\ca$ (the source algebra). To 
define an action functional one needs
noncommutative  generalizations of the
conformal and Riemannian geometries. 
Following 
\cite[Section VI.2]{co94}, 
conformal structures can be understood within the framework of positive 
Hochschild cohomology. In the
commutative case  the tri-linear map $\Phi: \ca\otimes \ca\otimes \ca \to \IR$ defined by
\beq\label{trilinear}
\Phi(f_{0},f_{1},f_{2})= \frac{\ii}{\pi} \int_{\Sigma}f_{0}\partial
f_{1}\wedge\overline{\partial}f_{2}
\eeq is an extremal element of the space of positive 
Hochschild cocycles \cite{CoCu} and
belongs to the Hochschild cohomology class of the cyclic cocycle 
$\Psi$ given by
\beq\label{trilinearbis}
\Psi(f_{0},f_{1},f_{2})= \frac{\ii}{2\pi}\int_{\Sigma}f_{0} \dd f_{1} 
\wedge \dd f_{2}.
\eeq
Clearly, both \eqref{trilinear} and \eqref{trilinearbis} still make 
perfectly sense for a general
noncommutative algebra $\ca$. One can say that
$\Psi$ allows to integrate 2-forms in dimension 2, 
$$
\frac{\ii}{2\pi} \int
a_{0} \dd a_{1}\dd a_{2}=\Psi(a_{0},a_{1},a_{2}) 
$$
so that it is a metric independent
object, whereas $\Phi$ defines a suitable scalar product
$$
\langle a_{0}\dd a_{1},b_{0}\dd b_{1}
\rangle=\Phi(b_{0}^{*}a_{0},a_{1},b_{1}^{*})
$$
on the space of 1-forms and thus
depends on the conformal class of the metric. Furthermore, this 
scalar product is positive and its relation to the cyclic cocycle $\Psi$ allows one to 
prove various
inequalities involving topological quantities. In particular, in typical models one 
gets a topological lower bound for the
action which is a two dimensional analogue of the inequality in 
four dimensional Yang-Mills theory leading to self-duality equations.

As mentioned, the simplest example of a target space is that of a 
finite space made of two points $M = \{1,2\}$. Any continuous
map from a connected surface to a discrete space is constant and the resulting
(commutative) theory would be trivial. This is no longer 
true if the source is a
noncommutative space and one has, in general, lots of such maps (i.e. algebra
morphisms).
Now, the algebra of functions over $M = \{1,2\}$ is just $\cb = 
\IC^{2}$ and any
element $f\in\cb$ is a couple of complex numbers $(f_1,f_2)$ with 
$f_a=f(a)$, the value of $f$
at the point $a$. As a vector space $\cb$ is generated by 
a hermitian idempotent (a
projection), $e^2 = e^* = e$, that is the function defined by $e(1)=1$ and $e(2)=0$. 
Thus, any $*$-algebra morphism $\pi$ from $\cb$ to $\ca$ is given by a 
projection ${\qp}=\pi(e)$ in $\ca$. 

With a given Hochschild cocycle $\Phi$ standing for the conformal structure,
the action functional on the collection of all projection in $\ca$ can be taken to be
$$
S[\qp]=\Phi(1,\qp,\qp)   .
$$
And as already mentioned, from general consideration of positivity in Hochschild cohomology 
this action  is bounded by a topological term of the form 
$$
\mbox{Top}(\qp) = \Psi(\qp,\qp,\qp)   ,
$$
for $\Psi$ a cyclic cocycle which is Hochschild-cohomologous to the Hochschild cocycle $\Phi$. 

In the initial papers \cite{DKL00,DKL03} there were already 
models for the source algebra be the algebra 
$\ca=C^\infty(\IT^2_\theta)$ of the noncommutative torus. While the equations for the critical points were rather complicate, the use of the positivity and the bound by the topological term lead to self-duality equations for projections in $C^\infty(\IT^2_\theta)$, with solutions being topological non-trivial (thus eligible to be called noncommutative solitons, or noncommutative instantons). The construction relied on complex structures and the existence of holomorphic structures for projective modules over noncommutative spaces. The projections in the source algebra were coming from Morita equivalence bimodules. 
Some further generalization were given in \cite{DLL15} where  several results from time-frequency analysis and Gabor analysis were used to construct new classes of sigma-model solitons over the Moyal plane and over noncommutative tori. And some related results are in \cite{MR11} 

In the present paper we go beyond the scheme above by considering a twisted Hochschild cocycle $\Phi_\sigma$ for the action functional, and a twisted cyclic cocycle $\Psi_\sigma$
for the `topological' term. This will still be Hochschild-cohomologous to the twisted Hochschild cocycle and from positivity in twisted Hochschild cohomology there is again a bound on the action. Now, the equations for the critical points will be even more complicate, but again the use of the positivity and the bound by the topological term will lead to self-duality equations for projections (thus leading to twisted noncommutative solitons, or instantons). 

As we shall explain in details, for a projection $\qp \in \Mat_{N}(\ca)$, with $\ca$ the `source space' algebra, we shall have the action functional, the energy of the projection $\qp$,  
$$
S_\sigma(\qp) = \frac{1}{2} \int_h \qtr(\qp\, (\star \dd \qp) \wedge \dd \qp)   ,
$$
as well as a `topological' term (the $q$-index of a Dirac operator)
$$
\mbox{Top}_\sigma(\qp) = \frac{1}{2} \int_h \qtr(\qp\, \dd \qp \wedge \dd \qp)   .
$$
Here $\qtr$ is a `quantum' trace, $\star$ is a Hodge operator and $h$ is a positive twisted state 
on $\ca$ with modular automorphism $\sigma \in \Aut(\ca)$. Both functionals above are twisted by this modular automorphism. There is still a bound
$$
S_\sigma(\qp) \geq | \mbox{Top}_\sigma(\qp) | \,
$$
 with equality occurring for the projector $\qp$ satisfying equations
$$
 \star \qp\, \dd \qp = \pm \, \qp\, \dd \qp   .
$$
We shall exhibit explicit 
 solutions of these self-dual or anti-self-dual equations while naming them noncommutative 
 twisted solitons (or twisted sigma-model anti/instantons).  

In particular, 
we present models for $\ca=\Apq$ the coordinate algebra of the quantum projective line $\pq$ and $\qp \in \Mat_{N}(\Apq)$. 
Now $h$ is the restriction to $\pq$ of the Haar state of the quantum group $\SU$. 
All constructions rely on complex structures and the existence of holomorphic structures for projective modules over 
the quantum projective line $\pq$ as introduced and developed in \cite{KLvS}. The topological term is the $q$-index of a Dirac operator on $\pq$, and the duality equations are also written as
$$
 \qp \, \dol \qp = 0  \qquad \mbox{or} \qquad  \qp \, \dolb \qp = 0   ,
 $$
 for a complex structure $\dol$, $\dolb$ on $\pq$. As mentioned, we shall give explicit solutions of these equations. 
The results of \cite{KLvS} on complex and holomorphic structures on  
$\pq$ were generalised in \cite{KM11,KM11b} to the quantum projective spaces $\pqn$. These results when used for the projections on $\pqn$ constructed in \cite{DL10,DD10} would lead to a direct, if computationally intricate, generalisation of the results of the present paper.

\subsubsection*{Acknowledgments} 
Part of this work was done during a visit at the Tohoku Forum for Creativity in Sendai. I am grateful to Yoshi Maeda for the kind invitation and to him, to Ursula Carow--Watamura and to Satoshi Watamura for the great hospitality in Sendai. Francesco D'Andrea read a first version of the paper and made useful suggestions. 

\section{The geometry of the quantum projective line}\label{se:qhb}

The quantum projective line $\pq$
is the standard Podle\'s sphere $\sq$ of \cite{Po87} --- the quotient of the
sphere $\mathrm{S}^3_q \simeq\SU$ for an action of the abelian
group $\U(1)$ ---, 
with additional structure. This is the well-known
quantum principal $\U(1)$-bundle over $\sq$, whose total space is the
manifold of the quantum group $\SU$ as we now briefly recall. 
In the following we use the `$q$-number', defined for $q \neq 1$ and any $x \in \IR$ as
$$
[x] = [x]_q := \frac{q^x - q^{-x}}{q - q^{-1}}   .
$$

\subsection{The quantum sphere $\mathrm{S}^3_q$}\label{q3s}
The manifold $\mathrm{S}^3_q$ is though of as the 
manifold `underlying' the quantum group $\SU$. The coordinate algebra $\ASU$ of the latter \cite{wor87} is the $*$-algebra generated by elements $a$ and $c$ subject to relations
\begin{align*}
a\,c &= q\,c\,a \quad \mbox{and} \quad c^*\,a^* =q\,a^*\,c^*  \ ,     
\quad a\,c^* =q\,c^*\,a \ 
\quad \mbox{and} \quad c\,a^* =q\,a^*\,c  \ , \nn \\[4pt] 
c\,c^* &= c^*\,c \quad \mbox{and} \quad a^*\,a+c^*\,c =
a \, a^*+q^{2}\,c \, c^* =1 \ . 
\end{align*}
The deformation parameter
$q\in\IR$ can be restricted to the interval $0<q<1$ without loss of
generality.  
The well known Hopf algebra structure for $\ASU$, in terms of a coproduct $\Delta$, an antipode $S$ and a counit $\epsilon$ is not explicitly needed in the following. 

On the other hand, 
the algebra of symmetries, 
the quantum universal enveloping algebra $\su$, is the Hopf $*$-algebra
generated, as an algebra, by four elements $K,K^{-1},E,F$ such that 
$K\,K^{-1}=1=K^{-1}\,K$ and with relations
$$
K^{\pm\,1}\,E = q^{\pm\,1}\,E\,K^{\pm\,1} \ , \quad 
K^{\pm\,1}\,F = q^{\mp\,1}\,F\,K^{\pm\,1} \quad \mbox{and} \quad
[E,F]  = \frac{K^{2}-K^{-2}}{q-q^{-1}} \ . 
$$
The $*$-structure is 
$K^* = K$, $E^* = F$ and $F^* = E$. 
And the Hopf algebra structure (used later on) is given by the coproduct $\Delta$, 
the antipode $S$, and the counit $\epsilon$ defined by
$$
\begin{array}{cc}
\Delta(K^{\pm\,1})=K^{\pm\,1}\otimes K^{\pm\,1} \ ,
\quad \Delta(E)=E\otimes K+K^{-1}\otimes E \ , \quad
\Delta(F)=F\otimes K+K^{-1}\otimes F \ , \nn\\[4pt]
S(K)  = K^{-1} \ ,
\quad S(E)  = -q\,E \ , \quad S(F)  = -q^{-1}\,F \ , \nn\\[4pt]
\epsilon(K) = 1 \ , \quad \epsilon(E) = \epsilon(F) = 0 \ .
\end{array}
$$
There is a bilinear pairing between $\su$ and $\ASU$ given on
generators by
\begin{align*}
&\langle K^{\pm\,1}, a\rangle = q^{\mp\, 1/2} \quad \mbox{and} \quad
\langle K^{\pm\,1}, a^*\rangle = q^{\mp\, 1/2} \ , \nn\\[4pt]
&\langle E,c\rangle = 1 \quad \mbox{and} \quad \langle
F,c^*\rangle = -q^{-1} \ ,
\end{align*}
with all other couples of generators pairing to zero. 
This pairing makes $\su$ a subspace of the linear dual of $\ASU$.
And it yields canonical left and right $\su$-module algebra structures (actions) 
on $\ASU$ such that~\cite{wor87}
$$
\hs{g}{h \lt x} := \hs{g\,h}{x}\quad \mbox{and} \quad
\hs{g}{x \rt  h} := \hs{h\,g}{x}   ,
$$
for all $g,h \in \su,\ x \in \ASU$. 
The right and left actions are given by
$$
h \lt x := \co{x}{1} \,\hs{h}{\co{x}{2}} \quad \mbox{and} \quad
x \rt  h := \hs{h}{\co{x}{1}}\, \co{x}{2}   
$$
in the Sweedler notation for the coproduct.
These right and left actions commute:
\begin{align*}
(h \lt x) \rt  g  = \left(\co{x}{1} \,\hs{h}{\co{x}{2}}\right) \rt g 
= \hs{g}{\co{x}{1}} \,\co{x}{2}\, \hs{h}{\co{x}{3}} 
= h \lt \left(\hs{g}{\co{x}{1}}\, \co{x}{2}\right)  = h \lt (x \rt g) \ .
\end{align*}
Since the pairing satisfies
$\hs{S(h)^*}{x} = \overline{\hs{h}{x^*}}$,
the $*$-structure is compatible with 
both actions, that is, for all $h \in \su, \ x \in \ASU$ it holds that 
$$
h \lt x^*  =  \big(S(h)^* \lt x\big)^*  \quad \mbox{and} \quad
x^* \rt  h  =  \big(x \rt S(h)^*\big)^*
$$
For later use we list the left action.  With $s\in\IN$ a positive integer, it is worked out to be:\begin{align}\label{lact}
K^{\pm\,1}\triangleright a^{s}  = q^{\mp\,\frac{s}{2}}\,a^{s}   , 
&\qquad 
K^{\pm\,1}\triangleright c^{s}  = q^{\mp\,\frac{s}{2}}\,c^{s}   , \nn\\[4pt] 
K^{\pm\,1}\triangleright a^{*}\,^{s} = q^{\pm\,\frac{s}{2}}\,a^{*}\,^{s}   ,
&\qquad 
K^{\pm\,1}\triangleright c^{*}\,^{s} = q^{\pm\,\frac{s}{2}}\,c^{*}\,^{s} \ , \nn\\[4pt]
F\triangleright a^{s}  = 0   ,
&\qquad 
F\triangleright c^{s}  = 0   ,  \nn\\[4pt]
F\triangleright a^{*}\,^{s}  = q^{(1-s)/2} \,[s]\, c\, a^{*}\,^{ s-1} \ ,
&\qquad 
F\triangleright c^{*}\,^{s}  = -q^{-(1+s)/2}\, [s]\, a\, c^{*}\,^{s-1} \ , \nn \\[4pt]
E\triangleright a^{s}  = -q^{(3-s)/2}\, [s]\, a^{s-1}\, c^{*}   ,
&\qquad 
E\triangleright c^{s}  = q^{(1-s)/2}\, [s] \, c^{s-1}\, a^*   , \nn\\[4pt] 
E\triangleright a^{*}\,^{ s}  = 0 \ ,
&\qquad 
E\triangleright c^{*}\,^{ s}  = 0 \ .
\end{align}

\noindent
We also need the irreducible finite dimensional representations $\pi_J$ of $\su$.
They are labelled by nonnegative
half-integers $J\in \tfrac{1}{2}\IN$ (the spin) and given by (c.f. \cite[Sect.3.2.3]{KS98}).
\begin{align}\label{eq:uqsu2-repns}
\pi_J(K)\,\ket{J,m} &= q^m \,\ket{J,m},
\nn \\[4pt]
\pi_J(E)\,\ket{J,m} &= \sqrt{[J-m][J+m+1]} \,\ket{J,m+1}, \nn \\[4pt]
\pi_J(F)\,\ket{J,m} &= \sqrt{[J-m+1][J+m]} \,\ket{J,m-1},
\end{align}
where the vectors $\ket{J,m}$, for $m = J, J-1,\dots, -J+1, -J$, form
an orthonormal basis for the $(2J+1)$-dimensional, 
irreducible $\su$-module $V_J$.
With respect to the hermitian scalar product on $V_J$ for which the vectors $\ket{J,m}$ are
orthonormal, $\pi_J$ is a $*$-representation.   

\subsection{The quantum projective line $\pq$}\label{qdct}
Next we describe the well know $\U(1)$-principal bundle over the sphere $\sq$, 
with total space $\SU$, a quantum homogeneous space \cite{BM93}.  
The abelian group $\U(1)= \{ z\in \IC \, | \, z \,z^* = 1\}$ acts on the algebra $\ASU$,  with a map
$\alpha : \U(1) \longrightarrow \Aut\big(\ASU\big)$
defined on generators by  
\begin{align}\label{rco}
\alpha_z(a) = a \, z   ,  \quad  
\alpha_z(a^*) = a^*\, z^{*} \, \qquad \mbox{and}& \qquad
 \alpha_z(c) = c\,  z   , \quad  
\alpha_z(c^*) = c^*\, z^{*} \ , 
\end{align}
extended as an algebra map, $\alpha_z(x\, y) =
\alpha_z(x)\,\alpha_z(y)$ for $x,y\in \ASU$ and $z\in \U(1)$. 
The coordinate algebra $\Asq$ is then the
subalgebra of \emph{invariant} elements in $\ASU$: 
$$
\Asq := \ASU^{\U(1)} := \big\{ x \in\ASU  \; \big| \; \alpha_z(x)= x
\big\} \ .
$$
As a set of generators for $\Asq$ one may take 
$$
B_{-} := a\,c^* \ , \qquad 
B_{+} := c\,a^* \quad \mbox{and} \quad 
B_{0} := c\,c^* \ ,
$$
for which one finds relations
\begin{align*}
B_{-}\,B_{0} &= q^{2}\, B_{0}\,B_{-} \quad \mbox{and} \quad
B_{+}\,B_{0} = q^{-2}\, B_{0}\,B_{+} \ , \\[4pt]
B_{-}\,B_{+} &= q^2\, B_{0} \,\big( 1 - q^2\, B_{0} \big) \qquad
\mbox{and} \qquad B_{+}\,B_{-}= B_{0} \,\big( 1 - B_{0} \big) \ ,
\end{align*}
and $*$-structure $(B_{0})^*=B_{0}$ and $(B_{+})^*= B_{-}$.
The algebra inclusion $\Asq\hookrightarrow\ASU$ is a quantum principal
bundle \cite{BM93} and can be endowed with compatible non universal calculi, a
construction that we shall illustrate later on.

In \S\ref{se:cals2} we will recall the complex structure on
the quantum two-sphere $\sq$ for the unique two-dimensional covariant
calculus on it. This will transform the sphere $\sq$ into a quantum
riemannian sphere or quantum projective line $\pq$. Having this in
mind, with a slight abuse of `language' we will speak of $\pq$ rather
than $\sq$ from now on.

%

\subsection{Equivariant line bundles on $\pq$}\label{se:avb} 
Let $\rho : \U(1) \to V$ be a representation of $\U(1)$ on a
finite-dimensional complex vector space $V$. 
The space defined by
$$
\ASU \boxtimes_\rho V := \big\{ \varphi  \in  \ASU \otimes V \, \big| \,  
(\alpha \otimes \id)\varphi = \big((\id \otimes \rho^{-1} )\big)
\varphi \big\} \ ,
$$
where $\alpha$ is the action \eqref{rco} of $\U(1)$ on $\ASU$,
is the space of $\rho$-equivariant elements, clearly an $\ca(\pq)$-bimodule. We shall think of
it as the module of sections of the vector bundle associated with the
quantum principal $\U(1)$-bundle on $\pq$ via the representation
$\rho$. There is a natural $\SU$-equivariance, in that the left
coaction $\Delta$ of $\ASU$ on itself extends in a natural way to a
left coaction on $\ASU \boxtimes_\rho V$ given by
$$
\Delta^\rho= \Delta \otimes \id\,:\, \ASU \boxtimes_\rho V
~\longrightarrow~ \ASU \otimes \big(\ASU \boxtimes_\rho V\big) \ .
$$

Now, the irreducible representations of $\U(1)$ are labelled by an integer
$n\in\IZ$. If $C_n\simeq\IC$ is the irreducible one-dimensional left
$\U(1)$-module of weight $n$, they are given by  
\beq\label{ircore}
\rho_n \,:\, \U(1) ~\longrightarrow~ \Aut(C_n) \ , \qquad C_n \ni v
~\longmapsto~ z^n\, v \in C_n \ .
\eeq
The spaces of $\rho_n$-equivariant elements yield a vector space decomposition~\cite[eq.~(1.10)]{maetal}
$$
\ASU=\bigoplus\nolimits_{n\in\IZ}\, \cl_n \ ,
$$
where
\beq\label{libu} 
\cl_n := \ASU \boxtimes_{\rho_n} \IC \simeq \big\{x \in \ASU ~\big|~
\alpha_z(x) = x \, (z^*)^n \big\} \ .
\eeq 

In particular $\cl_0 = \Apq$. One has $\cl_n^* =
\cl_{-n}$ and $\cl_n\,\cl_m = \cl_{n+m}$. Each $\cl_n$ is
clearly a bimodule over $\Apq$.
It may be given as a finitely-generated projective right (and left)
$\Apq$-module of rank one \cite[Prop.~6.4]{SWPod}, and one may think of it as giving sections
of line bundles over $\pq$. 

A dual presentation of the line bundles $\cl_n$ as 
\beq\label{libudual} 
\cl_n = \big\{x \in \ASU ~\big|~ K \lt x = q^{n/2}\, x \big\} \ .
\eeq 
uses the left action of the group-like element $K$ on $\ASU$. 
Indeed, if $H$ is the infinitesimal generator of the $\U(1)$-action $\alpha$, 
the group-like element $K$ can be written as $K = q^{-H/2}$.
Then, from the left action \eqref{lact} of $\su$ one finds
$$
E \lt \cl_n ~\subset~ \cl_{n+2} \quad \mbox{and} \quad
F \lt \cl_n ~\subset~ \cl_{n-2} \ .
$$
On the other hand, commutativity of the left and right actions of
$\su$ yields 
$$
\cl_n \rt  h \subset \cl_n
$$
for all $h\in \su$. More details on these modules will be given later on. For the moment we just mention the following results (cfr. \cite[Prop.~3.1]{KLvS}):
The natural map $\cl_n \otimes \cl_m \to \cl_{n+m}$ defined by
multiplication induces an isomorphism of $\Apq$-bimodules 
\beq\label{masu}
\cl_n \otimes_{\Apq} \cl_m \simeq \cl_{n+m} \ ,
\eeq
and in particular $\Hom_{\Apq}(\cl_m, \cl_n) \simeq\cl_{n-m}$.
Moreover, from the transformations in \eqref{rco}, it follows that an
$\Apq$-module generating set for $\cl_n$ is given by elements
\beq
(\Psi_{(n)})_{\mu}  
=  \begin{cases}
\, a^{*}\,^{ \mu}\,c^{*}\,^{ n-\mu} & \mathrm{for } 
\quad n\geq0 ~, \quad \mu = 0,1, \dots,  n \ , \\[4pt]
\, a^{\mn-\mu}\,c^{\mu} & \mathrm{for } \quad n\leq0 ~, \quad\mu = 0,1, \dots,  \mn \ .
\end{cases}
\label{qpro0}
\eeq
Then one writes equivariant elements as
\beq
\varphi_f  =  
\begin{cases}
 \, \sum_{\mu=0}^{n}\, a^{*}\,^{\mu}\, c^{*}\,^{n-\mu} \,
f_{\mu}  
 =  \sum_{\mu=0}^{n}\, \tilde f_{\mu}~a^{*}\,^{\mu}\, c^{*}\,^{n-\mu} 
& \mathrm{for} \quad n \geq 0 \ , \nn\\[10pt]
 \,  \sum_{\mu=0}^{\mn}\, a^{\mn-\mu}\, c^{\mu}\, f_{\mu}
 =  \sum_{\mu=0}^{\mn}\, \tilde f_{\mu}~a^{\mn-\mu}\, c^{\mu}
& \mathrm{for} \quad n \leq 0  \ , 
\end{cases} 
\label{eqmap}
\eeq
with $f_\mu$ and $\tilde f_\mu$ generic elements in $\Apq$. The
bimodules $\cl_n$ are not free modules, thus the elements in \eqref{qpro} are not independent over $\Apq$.

A generic finite-dimensional representation $(V,\rho)$ for $\U(1)$ 
has a weight decomposition 
$$
V  =  \bigoplus_{n\in W(V)}\, C_n  \otimes V_n \ , \qquad \rho  = 
\bigoplus_{n\in W(V)}\, \rho_n \otimes  \id \ .
$$
Here $(C_n, \rho_n)$ is the one-dimensional irreducible representation
of $\U(1)$ with weight $n\in\IZ$ given in \eqref{ircore}, the spaces
$V_n=\Hom_{\U(1)}(C_n,V)$ are the multiplicity spaces, and the set
$W(V) = \{n\in\IZ ~|~V_n \not= 0\}$ is the set of weights of
$V$. For the corresponding space of $\rho$-equivariant elements we
have a corresponding decomposition 
$$
\ASU \boxtimes_\rho V = \bigoplus_{n\in W(V)}\, \cl_n  \otimes V_n \ ,
$$
with $\cl_n$ the irreducible modules in \eqref{libu} giving sections
of line bundles over $\pq$.

\subsection{Covariant calculi}\label{se:cotqpb}
The principal bundle $(\ASU,\Apq,\U(1))$
is endowed with compatible
non-universal calculi obtained from the
three-dimensional left-covariant calculus on the quantum group 
which we present first. We then describe the unique left-covariant
two-dimensional calculus on the quantum projective line
$\pq$ obtained by restriction, and also the projected
calculus on the structure group $\U(1)$. The calculus on $\pq$ can be
canonically decomposed into a holomorphic and an anti-holomorphic

\subsection{Left-covariant forms on $\mathrm{S}^3_q$}\label{se:lcc} 
The three-dimensional left-covariant calculus on the quantum group $\SU$
has quantum tangent space generated by three elements \cite{wor87}:
$$
X_{z}  = \frac{1-K^{4}}{1-q^{-2}} \ , \qquad X_{-}
 = q^{-1/2}\,F\,K  \quad \mbox{and} \quad X_{+}  = q^{1/2}\,E\,K =
 X_{-}^* \ .
$$
Their coproducts and antipodes are easily found to be
\begin{align}
\cop(X_z)  =  1\otimes X_z + X_z \otimes K^4 \quad &\mbox{and} \quad
\cop(X_\pm)  =  1\otimes X_\pm + X_\pm \otimes K^2 \ , \label{cotb}  \\[4pt]
S(X_z) = - X_z \, K^{-4} \quad &\mbox{and} \quad S(X_\pm) = - X_\pm
\, K^{-2} \ .  \label{antb}
\end{align}  
The dual space of one-forms $\Omega^1(\SU)$ has a basis 
$$
\omega_{z}  = a^*~\dd a+c^*~\dd c \ , \qquad \omega_{-}  = c^*~\dd
a^*-q\,a^*~\dd c^* \quad \mbox{and} \quad
\omega_{+}  = a~\dd c-q\,c~\dd a 
$$
of left-invariant forms. The differential $\dd: \ASU \to
\Omega^1(\SU)$ is given by
\beq\label{exts3} 
\dd f =
(X_{-}\triangleright f) \,\omega_{-} + (X_{+}\triangleright f)
\,\omega_{+} + (X_{z}\triangleright f) \,\omega_{z} 
\eeq 
for all $f\in\ASU$. If $\Delta^{(1)}$ is the (left) coaction of $\ASU$
on itself extended to forms, the left-coinvariance of the basis forms 
is the statement that $\Delta^{(1)}(\omega_{s})=1\otimes\omega_{s}$, 
while the left-covariance of the calculus is stated as
$$
(\Delta \otimes \id)\circ \Delta^{(1)} = \big(\Delta^{(1)} \otimes
\id\big)\circ \Delta^{(1)} \qquad \mathrm{and} \qquad  
(\epsilon \otimes \id)\circ \Delta^{(1)} = 1 \ .
$$
The requirement that it is a $*$-calculus, i.e. $\dd (f^*) = (\dd f)^*
$, yields
$$
\omega_{-}^*= - \omega_{+}  \quad \mbox{and} \quad
\omega_{z}^*=-\omega_{z} \ .
$$
The bimodule structure is given by
\begin{align}\label{bi1}
\omega_{z}\,a = q^{-2}\,a\,\omega_{z} \ , \quad
\omega_{z}\,a^* = q^{2}\,a^*\,\omega_{z} \ , \quad 
\omega_{\pm}\,a = q^{-1}\,a\,\omega_{\pm} \quad \mbox{and} \quad
\omega_{\pm}\,a^* = q\,a^*\,\omega_{\pm} \ , \nn \\[4pt]
\omega_{z}\,c = q^{-2}\,c\,\omega_{z} \ , \quad
\omega_{z}\,c^* = q^{2}\,c^*\,\omega_{z} \ , \quad 
\omega_{\pm}\,c = q^{-1}\,c\,\omega_{\pm} \quad \mbox{and} \quad
\omega_{\pm}\,c^* = q\,c^*\,\omega_{\pm} \ .
\end{align}
Higher degree forms can be defined in a natural way by requiring
compatibility with the commutation relations (the bimodule structure
\eqref{bi1}) and that $\dd^2=0$. 
One has
\beq \label{dformc3} \ 
\dd\omega_{z}  = -\omega_{-}\wedge\omega_{+} \ , \qquad 
\dd \omega_{\pm} = \pm q^{\pm 2}\,\big(1+q^{2}\big)\,\omega_{z}\wedge\omega_{\pm}   ,
\eeq 
together with the commutation relations
\begin{align}
\omega_{+}\wedge\omega_{+} = \omega_{-} & \wedge\omega_{-} = \omega_{z}
\wedge\omega_{z}=0 \ , \nn\\[4pt]
\omega_{-}\wedge\omega_{+}+q^{-2}\,\omega_{+}\wedge\omega_{-}=0 \ ,
\qquad &
\omega_{z}\wedge\omega_{\pm}+
q^{\mp 4}\,\omega_{\pm}\wedge\omega_{z} =0 \ .
\label{commc3}
\end{align}
Finally, there is a unique top form
$\omega_{-}\wedge\omega_{+}\wedge\omega_{z}$ commuting will all elements in $\ASU$.

\subsection{Holomorphic forms on $\pq$}\label{se:cals2}
Restricting the three-dimensional calculus of \S\ref{se:lcc} to
the quantum projective line $\pq$ yields the unique left-covariant
two-dimensional calculus on $\pq$ \cite{Po89}.
The `cotangent bundle' $\Omega^1(\pq)$ is
shown to be isomorphic to the direct sum $\cl_{-2}\oplus\cl_{+2}$ of the
line bundles with degree $\pm\, 2$. 

Since the element $K$ acts (on the left) as the identity on $\Apq$, the differential
\eqref{exts3} when restricted to $\Apq$ becomes
$$
\dd f  =  (X_{-}\triangleright f) \,\omega_{-} + (X_{+}\triangleright
f) \,\omega_{+} = q^{-1/2} (F \triangleright f) \,\omega_{-} + q^{1/2} (E \triangleright f) \,\omega_{+}
$$
for $f\in\Apq$. This leads to a decomposition of the exterior
differential into a holomorphic and an anti-holomorphic part, $\dd
=\dolb + \dol$, with 
\beq\label{hold}
\dolb f = \left(X_{-}\triangleright f\right)\,\omega_{-} \qquad
\mbox{and} \qquad
\dol f = \left(X_{+}\triangleright f\right)\,\omega_{+} 
\eeq
for $f\in\Apq$. 
The calculus on $\pq$ has then quantum tangent space generated by the two elements $X_+$ and $X_-$. 
Moreover, the two differentials $\dol$ and $\dolb$ satisfy
\beq\label{*d}
(\dol f)^* = \dolb f^*.
\eeq
Indeed, using the second relation in \eqref{antb} and compatibility of the $*$-structure with the left action, by direct computation:
$$
(\dol f)^* = - \omega_- (X_+ \lt f)^* = - \omega_- (S(X_+)^* \lt f^*= \omega_- (K^{-2} X_- \lt f^*) = 
(X_- \lt f^*) \omega_- = \dolb f^* ,
$$
This also shows that 
\beq\label{xstar}
(X_+ \lt f)^* = - q^2 X_- \lt f^*  \quad \mbox{and} \quad (X_- \lt f)^* = - q^{-2} X_+ \lt f^*
\eeq
a result we need later on. An explicit computation on the generators of $\pq$ yields 
\begin{align*}
\dolb \bm =  -q^{-1} \, a^{2} \, \omega_{-} \ , \qquad
\dolb\bz = -q^{-1} \, c\,a  \, \omega_{-} \qquad &\mbox{and} \qquad
\dolb\bp =  c^{2} \, \omega_{-} \ , \nn\\[4pt]
\dol\bp =  q \, a^{*}\,^{2} \, \omega_{+} \ , \qquad
\dol\bz =  c^{*}\,a^{*} \, \omega_{+} \qquad &\mbox{and}
\qquad \dol\bm =  -q^{2}\, c^{*}\,^{2} \, \omega_{+} \ .
\end{align*}
And it follows that the module of one-forms decomposes as 
$$
\Omega^1(\pq)=\Omega^{0,1}(\pq) \oplus
\Omega^{1,0}(\pq) \ ,
$$ 
where $\Omega^{0,1}(\pq)\simeq\cl_{-2} \omega_- \simeq\dolb(\Apq)$ is
the $\Apq$-bimodule generated by
$$
\big\{\,\dolb\bm\,,\,\dolb\bz\,,\,\dolb\bp
\big\} = \big\{a^{2}\,,\,c\,a\,,\,c^{2}\big\}\,\omega_{-}  = 
q^{2}\,\omega_{-}\,\big\{a^{2}\,,\,c\,a\,,\,c^{2}\big\}
$$
and $\Omega^{1,0}(\pq)\simeq\cl_{+2} \omega_+ \simeq\dol(\Apq)$ is the
$\Apq$-bimodule generated by
$$
\big\{\dol\bp\,,\,\dol\bz\,,\,\dol\bm\big
\} = \big\{a^{*}\,^{2}\,,\,c^{*}\,a^{*}\,,\,c^{*}\,^{2}\big\}\,
\omega_{+} =  q^{-2}\, \omega_{+}
\,\big\{a^{*}\,^{2}\,,\,c^{*}\,a^{*}\,,\,c^{*}\,^{2}\big\} \ .
$$
These modules of forms are not free and there are relations among the differentials:
$$
\dol\bz - q^{-2}\, \bm~\dol\bp + q^{2}\, \bp~\dol\bm = 0 
\quad \mbox{and} \quad \dolb\bz - \bp~\dolb\bm + q^{-4}\, \bm~\dolb\bp = 0 \ .
$$

The two-dimensional calculus on $\pq$ has a unique (up to scale) top invariant form
$\beta$, which is central, $\beta \,f = f\, \beta$ for all $f\in\Apq$, and
$\Omega^2(\pq)$ is the free $\Apq$-bimodule generated by $\beta$.
Both $\omega_{\pm}$ commute with elements of $\Apq$ and so does
$\omega_{-}\wedge\omega_{+}$, which may be taken as the 
generator $\beta=\omega_{-}\wedge\omega_{+}$ of
$\Omega^2(\pq)$ (cfr. \cite{maj05} or \cite[App.]{SW04}). 
As a model for a more general construction, we may summarise these results as follows.
\begin{prop}\label{2dsph}
The two-dimensional covariant differential calculus on the quantum projective
line $\pq$ is given by
$$
\Omega^{\bullet}(\pq) = \Apq \, \oplus \, \left(\cl_{-2} \omega_{-}
\oplus \cl_{+2} \omega_{+} \right) \, \oplus \, \Apq \omega_{-}\wedge\omega_{+} \ .
$$
Moreover, the splitting $\Omega^1(\pq) = \Omega^{1,0}(\pq)
\oplus\Omega^{0,1}(\pq)$, together with the two maps $\dol$ and
$\dolb$ given above, constitute a complex structure for the
differential calculus.
\end{prop}
Writing any one-form as $\alpha = x\,
\omega_{-} + y\, \omega_{+} \in \cl_{-2} \omega_{-} \oplus \cl_{+2}
\omega_{+}$, the product is given by
$$
(x \,\omega_{-} + y \,\omega_{+} ) \wedge (t \,\omega_{-} + z\, \omega_{+}
) = \big( x\, z - q^{ 2}\, y\, t \big)\, \omega_{-}\wedge\omega_{+} \ .
$$
By \eqref{dformc3} it is natural (and consistent) to demand $\dd \omega_-=\dd
\omega_+=0$ when restricted to $\pq$. Since $K$ acts as $q^{\pm\,1}$ on $\cl_{\pm\,2}$,  
the exterior derivative of a one-form is written as
\begin{align*}
\dd \alpha & = \dd (x \,\omega_{-} + y\, \omega_{+}) 
= \dol x \wedge \omega_{-} + \dolb y \wedge \omega_{+} \nn\\[4pt]
&=
 \big( X_- \lt y - q^{2}\, X_+ \lt x \big) \,
\omega_{-}\wedge\omega_{+} \ ,
\end{align*}
Notice that in this expression, both $X_+ \lt x$ and $X_- \lt y$ belong to $\Apq$, as they should. 

Finally, a Hodge operator at the level of one-forms is constructed in
\cite{maj05} via a left-covariant map $\star:\Omega^{1}(\pq) \to
\Omega^{1}(\pq)$ which squares to the identity $\id$. In the
description of the calculus as given in Proposition~\ref{2dsph}, it is defined
for all $f\in \Apq$, by
\beq\label{hod1} 
\star (\dol f) = \dol f  \quad \mbox{and} \quad \star ( \dolb f) = -\dolb f   .
\eeq 
One shows its compatibility with the
bimodule structure, i.e. the map $\star$ is a bimodule map. Thus
$\star$ has values $\pm\, 1$ on holomorphic or anti-holomorphic
one-forms respectively.
In particular, $\star 
\omega_\pm = \pm\, \omega_\pm$. The 
Hodge operator is naturally extended by requiring
$$
\star 1 = \omega_{-}\wedge\omega_{+} \quad \mbox{and} \quad
\star (\omega_{-}\wedge\omega_{+}) = 1 \ ,
$$
although this extension will not be needed in the present paper.

For completeness, we finally mention  the calculus on $\U(1)$ which
makes all three calculi compatible from the quantum principal bundle
point of view. One finds \cite{BM93} that the calculus is one-dimensional and bicovariant with quantum tangent space generated by
\beq\label{vvf} 
X = X_{z}  = \frac{1-K^{4}}{1-q^{-2}}   , 
\eeq 
with dual one-form $\omega_{z} = z^{*} \dd z$ and commutation relations
$$
\omega_{z} \,z  =  q^{-2}\, z\,\omega_{z} \ , 
\quad \mbox{and} \quad
\omega_{z}\, z^{*}  =  q^{2}\,z^{*}\,\omega_{z} \ .
$$

\section{Positive cocycles and twisted conformal structures}
In \cite[Section VI.2]{co94} it is shown that positive Hochschild cocycles on
the algebra of smooth functions on a compact oriented 2-dimensional manifold encode 
the information needed to define  a complex structure on the surface. 
The relevant positive cocycle  is in the same Hochschild cohomology class of the cyclic 
cocycle giving the fundamental class of the manifold.  This result 
suggests regarding positive cyclic and Hochschild 
cocycles as a starting point in defining complex noncommutative structures.

\subsection{Twisted positive Hochschild cocycles}\label{se:tphc}
On the quantum projective line there is no non-trivial 2-dimensional cyclic cocycles \cite{MNW}.  
In \cite{KLvS} we considered twisted Hochschild and cyclic cocycles and gave a notion of \emph{twisted positivity} for Hochschild cocycles. We exhibited an example for the complex structure on the quantum projective line.   

A Hochschild $2n$-cocycle $\varphi$ on an $*$-algebra $A$ is called 
{\it positive} \cite{CoCu} if the  pairing 
$$
\langle a_0\otimes a_1\otimes \cdots\otimes a_n, \,
  b_0\otimes b_1\otimes \cdots\otimes b_n \rangle = \varphi (
 b_0^* a_0, a_1, \cdots a_n, b_n^*,  \cdots,  b_1^*)   ,
 $$
defines a positive sesquilinear form on the vector spaces $A^{\otimes (n+1)}$. 
For $n=0$ this is a positive trace on an $*$-algebra. 
Given a differential graded $*$-algebra $(\Omega A, \, \dd)$ on $A$, 
a Hochschild $2n$-cocycle on $A$ defines a sesquilinear pairing on the space $\Omega^n A$ of $n$-forms (typically middle-degree ones). For $\omega=a_0 \dd a_1 \cdots \dd a_n$ and $\eta=b_0 \dd b_1 \cdots \dd b_n$ one defines 
$$ 
\langle \omega, \, \eta\rangle: = \varphi (b_0^* a_0, a_1, \cdots a_n, b_n^*,  \cdots,  b_1^*)   ,
$$
and extended by linearity. One has that $\langle a \omega, \, \eta \rangle = \langle  \omega, \, a^* \eta \rangle$ for all $a\in A$.
Positivity of $\varphi$ is equivalent to positivity of this sesquilinear form on $\Omega^n A$.

Next, let $A$ be an algebra and $\sigma: A \to A$ be an automorphism of $A$.
For $n \geq 0$,  let $C^n(A)= \Hom_{\mathbb{C}} (A^{\otimes (n+1)}, \mathbb{C})$ be the
space of $(n+1)$-linear functionals (the $n$-cochains) on $A$. Define the twisted cyclic
operator $\lambda_{\sigma}: C^n (A) \to C^n(A)$  by
$$ 
(\lambda_{\sigma} \varphi) (a_0, \cdots, a_n):= (-1)^n  \varphi (\sigma
(a_n), a_0, a_1, \cdots, a_{n-1}).$$
Clearly, $ (\lambda_{\sigma}^{n+1} \varphi) (a_0, \cdots, a_n)= \varphi
(\sigma(a_0), \cdots, \sigma (a_n) ) .$
Let 
$$
C^n_{\sigma} (A) = \text{ker} \left ( (1-\lambda_{\sigma}^{ n+1 }):
C^n (A) \longrightarrow C^n(A) \right) 
$$
denote the space of {\it twisted Hochschild n-cochains} on $A$. Define
the {\it twisted Hochschild coboundary} operator  $b_{\sigma}:
C^n_{\sigma} (A) \to  C^{n+1}_{\sigma} (A)$   and the operator $b_{\sigma}':
C^n_{\sigma} (A) \to  C^{n+1}_{\sigma} (A)$ by
\begin{align*}
b_{\sigma} \varphi (a_0, \cdots, a_{n+1}) &= \sum_{i=0}^n (-1)^{i}\varphi (a_0, \cdots,
a_ia_{i+1}, \cdots, a_{n+1}) \\
& \qquad\qquad\qquad\qquad\qquad + (-1)^{n+1}\varphi (\sigma (a_{n+1}) a_0, a_1, \cdots, a_{n-1}),\\
b_{\sigma}' \varphi (a_0, \cdots, a_{n+1}) &= \sum_{i=0}^n (-1)^{i}\varphi (a_0, \cdots,
a_ia_{i+1}, \cdots, a_{n+1})   .
\end{align*}
Then $b_{\sigma}$ sends twisted cochains to twisted
cochains. The cohomology of the complex $(C^*_{\sigma} (A), b_{\sigma})$
is the  {\it twisted Hochschild cohomology} of $A$.
The $n$-cochain $\varphi \in C^n (A)$ is called {\it twisted cyclic} if $(1-\lambda_{\sigma}) \varphi =0$, or, equivalently
$$
\varphi (\sigma (a_n), a_0, \cdots, a_{n-1})= (-1)^n  \varphi (a_0,
a_1, \cdots, a_{n}) ,
$$
$a_0, a_1, \dots, a_n$ elements in $A$. 
Denote the space of cyclic $n$-cochains by $C^n_{\lambda, \sigma} (A)$.
Obviously $C^n_{\lambda, \sigma} (A) \subset C^n_{\sigma} (A)$. The
relation $ (1-\lambda_{\sigma}) b_{\sigma}  = b'_{\sigma}
(1-\lambda_{\sigma})$ shows that the operator $b_{\sigma}$ preserves the
space of cyclic cochains and one gets the twisted cyclic complex of the
pair $(A, \sigma)$, denoted    $(C^n_{\lambda_\sigma} (A), b_{\sigma})$. The
cohomology of the  twisted cyclic complex, denoted $\mathrm{HC}^{\bullet}_\sigma(A)$, is  
the {\it twisted cyclic cohomology} of $A$. The following was given in \cite{KLvS}.
\begin{defi}
A twisted Hochschild $2n$-cocycle $\varphi$ on an $*$-algebra $A$ is called {\it twisted positive} if the  
pairing:
$$\langle a_0\otimes a_1\otimes \cdots\otimes a_n, \,
  b_0\otimes b_1\otimes \cdots\otimes b_n \rangle := \varphi (
 \sigma(b_0^*) a_0, a_1, \cdots a_n, b_n^*,  \cdots,  b_1^*).$$
 defines a positive sesquilinear form on the vector space $A^{\otimes (n+1)}$.
\end{defi} 
\noindent
For $n=0$ this is a twisted positive trace on an $*$-algebra. 
Note that we do not require $\varphi$ to be a twisted cyclic cocycle, only a Hochschild one. 

\subsection{A cocycle for the quantum projective line}\label{se:tphcs2}
Let us go back to the quantum projective line $\pq$.
Let $ h: \ASU \to \IC$ be the normalized positive Haar state of $\SU$. 
It is a positive twisted trace,  
$$ 
h(x y)= h(\sigma (y)x) , \qquad \textup{for} \quad x, y \in \ASU , 
$$ 
with (modular) automorphism $\sigma: \ASU \to \ASU$ given by
$$ 
\sigma (x)=K^2 \lt x \rt K^2 
$$
(cf. \cite[Prop.~4.15]{KS98}). When restricted to $\Apq$, this induces the automorphism  
\beq\label{sigmah}
\sigma: \Apq \to \Apq, \qquad \sigma (x)= x \rt K^2  , \qquad \textup{for} \quad x \in \Apq .
\eeq
Bi-invariance on $\ASU$ becomes right-invariance for $\su$ acting on $\Apq$:
$$
h(x \rt v) = \varepsilon(v) h(x) , \qquad \textup{for} \quad x\in \Apq ,  \; v \in \su .
$$

With $\omega_{-}\wedge\omega_{+}$ the central generator of $\Omega^2({\pq})$, $h$ the Haar state on $\Apq$, and $\sigma$ its modular automorphism \eqref{sigmah}, the linear functional 
$$
\int_h  : \;\; \Omega^2({\pq}) \to \IC, \qquad \int_h a \, \omega_{-}\wedge\omega_{+} := h(a), 
$$
defines \cite{SW04} a non-trivial twisted cyclic $2$-cocycle $\tau$ on $\Apq$ by 
$$
\tau(a_0,a_1,a_2):= \frac{1}{2} \int_h a_0\, \dd a_1 \wedge \dd a_2 .
$$
The non-triviality means that there is no twisted cyclic 1-cochain $\alpha$ on $\Apq$ such that
$b_\sigma \alpha = \tau$ and $\lambda_\sigma \alpha= \alpha$.  
Thus $\tau$ is a non-trivial class in $\mathrm{HC}^2_\sigma(\pq)$.

\begin{lemm} \label{lem0}
For any $a_0, a_1, a_2 , a_3 \in \Apq$ it holds that:
$$
\int_h a_0 ( \dol a_1 \dolb a_2 ) a_3 = \int_h \sigma (a_3) a_0  \dol a_1 \dolb a_2 .
$$
\end{lemm}
\begin{proof}
Write $ \dol a_1 \dolb a_2 = y \, \omega_- \wedge \omega_+$, for some $y \in \Apq$. 
Using the fact that 
$ \omega_- \wedge \omega_+$ commutes with elements in $\Apq$, we have that 
\begin{align*}
\int_h a_0 ( \dol a_1 \dolb a_2 ) a_3 - \int_h \sigma (a_3) a_0  \dol a_1 \dolb a_2 &= \int_h a_0 y \, \omega_- \wedge \omega_{+} a_3 - \int_h \sigma (a_3) a_0 y \, \omega_- \wedge \omega_{+}  \\
& = \int_h a_0 y a_3 \, \omega_- \wedge \omega_+  - \int_h \sigma (a_3) a_0 y \, \omega_- \wedge \omega_+  \\
& = h (a_0 ya_3)- h(\sigma (a_3) a_0 y)= 0  
 \end{align*}
from the twisted property of the Haar state.
\end{proof}

\begin{prop}\label{tpos}
The cochain $\varphi \in C^2 (\Apq)$ defined by 
$$
\varphi (a_0, a_1, a_2)=\int_h a_0\, \dol a_1 \, \dolb a_2
$$
is a twisted Hochschild 2-cocycle on $\Apq$, that is to say $b_{\sigma} \varphi =0$ and  
$\lambda^3_\sigma \varphi = \varphi$; it is also positive, with positivity expressed as:
$$ 
\int_h a_0\, \dol a_1 (a_0 \, \dol a_1)^* \geq 0   , \qquad \mbox{for all} \quad a_0, a_1 \in \Apq   .
$$
\end{prop}
\begin{proof}
Here we only show explicitly the twisted positivity for later use while referring to \cite{KLvS} for the full proof. Now, the hermitian scalar product on $\Omega^{(1,0)}(\pq)$, 
$$
\langle  a_0  \partial a_1, b^0  \partial b^1 \rangle: = \varphi(\sigma(b_0^*) a_0, a_1, b_1^*) =
\int_h \sigma(b_0^*) a_0 \, \dol a_1\, \dolb b_1^*,
$$
determines a positive sesquilinear form if for all $a_0, a_1 \in \Apq$ it holds that
$$ 
\int_h \sigma(a_0^*) a_0  \dol a_1\, \dolb a_1^* = \int_h a_0 \, \dol a_1 (a_0 \, \dol a_1)^* \geq 0 .
$$
The first equality follows again from Lemma~\ref{lem0}. Indeed,  
\begin{align*} 
\int_h a_0 \dol a_1 (a_0  \partial a_1)^*   = \int_h a_0 \dol a_1 (\partial a_1)^* a_0^* =
\int_h \sigma(a_0^*) a_0 \dol a_1 \dolb a_1^*. 
\end{align*} 
Then, if $ \dol a_1= y\omega_+$ it follows that  $ \dolb a_1^* = (\dol a_1)^* = - \omega_- y^*$ from \eqref{*d}. In turn, 
\begin{align*}
\int_h \sigma(a_0^*) a_0 \, \dol a_1\, \dolb a_1^* & = - \int_h \sigma(a_0^*) a_0 \, y\, \omega_+ \wedge \omega_- y^* 
= q^{2} \int_h \sigma(a_0^*) a_0 \, y\, y^* \, \omega_- \wedge \omega_+ \\[4pt] 
& = q^{2} h(\sigma(a_0^*) a_0 \, y y^*) = q^{2} h (a_0 y y^* (a_0)^*) = q^{2} 
h ( a_0 y  (a_0 y)^* ) \geq 0, 
\end{align*}
the positivity being then evident from the positivity of the state $h$.
\end{proof}
\noindent
In \cite{KLvS} it was also shown that the twisted Hochschild cocycles $\tau$ and $\varphi$ are cohomologous: $\varphi -\tau = b_{\sigma} \psi$, 
for the twisted Hochschild 1-cochain $\psi$ on $\Apq$ given by
$$
\psi (a, b) = \frac{1}{2}\, \int_h a \, \dol \dolb (b)   .
$$
Notice that the cocycle $\varphi$ is not twisted cyclic, only a twisted Hochschild one.
In fact, a trivial one: the twisted Hochschild cohomology of the algebra $\Apq$ is trivial for the for modular automorphism $\sigma$ \cite{had07}.

\subsection{Couplings to bundles}
To couple the twisted cyclic cocycle $\tau$ (and the twisted Hochschild cocycle $\varphi$) with  bundles over $\pq$, one needs a twisted Chern character. For the present paper it is enough to consider the lowest term, that of a twisted or `quantum trace' \cite{wa07}. 
Looking at $\sigma (x)= x \rt K^2$ in \eqref{sigmah}, if 
$M \in \Mat_{n+1}(\Apq)$, its (partial) quantum trace is the element $\qtr(M) \in \Apq$ given by
$$
\qtr(M):= \tr\left(M \pi_{n/2} (K^{2})\right) :=\sum\nolimits_{jl} M_{jl} 
\left(\pi_{n/2}(K^{2})\right)_{lj},
$$
where $\pi_{n/2} (K^{2})$ is the matrix from \eqref{eq:uqsu2-repns}
for the spin $J=n/2$ representation of $\su$. Notice that since the representation matrix $\pi_{n/2} (K^{2})$ is diagonal and positive, as it happens for the usual trace, for $M$ a positive matrix $\qtr(M)=0$ implies $M=0$, a fact we shall use later on. 
The $q$-trace is `twisted' by the automorphism $\sigma$, that is 
$$
\qtr(M_1 M_2 ) = \qtr\left( (M_2 \rt K^{2}) M_1 \right) = 
\qtr\left( \sigma(M_2) M_1 \right).
$$
Indeed, with `right crossed product' rules, that is 
$x h = {\co{h}{1}}\,(x \rt \co{h}{2})$, for $x\in\Apq$ and $h\in \su$ with Sweedler notation, 
one finds
\begin{align*}
\qtr(M_1 M_2 ) &= \tr\left( M_1 M_2 \pi_{n/2} (K^{2})\right) 
= \tr\left( M_1 \pi_{n/2} (K^{2}) (M_2 \rt K^{2}) \right) \\[4pt]
&= \tr\left((M_2 \rt K^{2}) M_1 \pi_{n/2} (K^{2})  \right) =
\qtr\left( (M_2 \rt K^{2}) M_1 \right) .
\end{align*}

Then, if $\qp \in \Mat_{n+1}(\Apq)$ is a projection, its pairing with the twisted cyclic cocycle 
$\tau$ via the twisted trace,
\beq\label{qind}
\big(\tau \circ \qtr\big) (\qp, \qp, \qp ) = \frac{1}{2} \int_h \qtr(\qp\, \dd \qp \wedge \dd \qp) , 
\eeq
results into a $q$-integer that is obtained \cite{NT05,wa07} as the $q$-index of the the Dirac operator on $\pq$, i.e. as the difference between the quantum dimensions of its kernel and
cokernel computed using $\qtr$. The quantity \eqref{qind}
may be regarded as a quantum Fredholm index computed from
the pairing between the $\sigma$-twisted cyclic cohomology and the
(Hopf algebraic) $\su$-equivariant K-theory of $\pq$. 

On the other hand, the pairing of $\qp$ with the twisted Hochschild cocycle $\varphi$,
\beq\label{qpos}
\big(\varphi \circ \qtr \big) (\qp, \qp, \qp ) = 
\int_h  \qtr(\qp\, \dol \qp \, \dolb \qp )
\eeq
is still positive from Proposition~\ref{tpos}, with positivity expressed now as:
$$
\int_h \qtr\big(\qp\, \dol \qp \, (\qp \, \dol \qp)^*\big) \geq 0   .
$$
Much structure will emerge by comparing the $q$-index in \eqref{qind} with the positive quantity in \eqref{qpos} as we show in the next section. 

\section{The action functional and the lower bound}
We arrive to a natural energy functional for a projection $\qp \in \Mat_{n+1}(\Apq)$ when 
comparing the $q$-index in \eqref{qind} with the positive quantity in \eqref{qpos}. The construction is valid for any algebra with a `twisted positive geometry' as we shall indicate. 

\subsection{Sigma-models on the quantum projective line}
As a working model, we start with the quantum projective line $\pq$ with the details of its geometry as described before. 

\begin{prop} \label{apeq} 
For any projection $\qp \in \Mat_{n+1}(\Apq)$ it holds that
\beq\label{ATP}
\int_h \qtr \big( \qp\, (\star \dd \qp) \wedge \dd \qp  \big) +
\int_h \qtr \big( \qp\, \dd \qp \wedge \dd \qp  \big)
= 2 \int_h \qtr\big(\qp\, \dol \qp \, (\qp \, \dol \qp)^*\big) 
\eeq
as well as 
\beq\label{ATN}
\int_h \qtr \big( \qp\, (\star \dd \qp) \wedge \dd \qp  \big) -
\int_h \qtr \big( \qp\, \dd \qp \wedge \dd \qp  \big)
= - 2 \int_h \qtr\big(\qp\, \dolb \qp \, (\qp \, \dolb \qp)^*\big) .
\eeq
\end{prop}
\begin{proof}
Since entries of $\qp$ are in $\Apq$ they are invariant by the left action of $K$ and commute with both one-forms $\omega_{\pm}$. Then, using $\qp^2 = \qp$ and Leibniz rule for $X_{\pm}$ one has:
\begin{align}
\qp \dol \qp \wedge ( \qp \dolb \qp)^* &= \qp (\dol \qp \wedge \dolb \qp) \qp 
= \qp (\dol \qp \wedge \dolb \qp)   ,   \qquad \label{P}  \\[4pt]
\mbox{as well as}  \qquad  \qquad
\qp \dolb \qp \wedge ( \qp \dol \qp)^* &= \qp (\dolb \qp \wedge \dol \qp) \qp 
= \qp (\dolb \qp \wedge \dol \qp)  .   \qquad  \qquad \label{N} 
\end{align}
Next,  
\begin{align}\label{T}
 \qp\, \dd \qp \wedge \dd \qp 
= \qp\, ( \dol \qp \wedge \dolb \qp + \dolb \qp \wedge \dol \qp)   ,
\end{align}
while, using also \eqref{hod1}, 
\begin{align}\label{A}
\qp\, (\star \dd \qp) \wedge \dd \qp  
= \qp\, ( \dol \qp \wedge \dolb \qp - \dolb \qp \wedge \dol \qp)   .
\end{align}
Then, a comparison of \eqref{T} and \eqref{A} with \eqref{P} yields \eqref{ATP}, while comparing them with \eqref{N} yields \eqref{ATN}. 
\end{proof}

\begin{prop}\label{bound}
For any projection $\qp \in \Mat_{n+1}(\Apq)$ it holds that
\beq\label{dis}
\int_h \qtr(\qp\, (\star \dd \qp) \wedge \dd \qp) \geq \left | \int_h \qtr(\qp\, \dd \qp \wedge \dd \qp) 
\right |
\eeq
with equality when 
 \beq\label{sde}
 \star \qp\, \dd \qp = \pm \, \qp\, \dd \qp   .
 \eeq
\end{prop}
\begin{proof}
In the expressions \eqref{ATP} and  \eqref{ATN} the right hand side is always greater or equal than zero. Then, if $\int_h \qtr(\qp\, \dd \qp \wedge \dd \qp) \geq 0$ the first statement follows from \eqref{ATN}. On the other hand, if $\int_h \qtr(\qp\, \dd \qp \wedge \dd \qp) \leq 0$
the statement follows from \eqref{ATP}. 
Indeed, the left hand sides of \eqref{ATN} and \eqref{ATP} above are 
$$
\int_h \qtr \big(\star \qp\, \dd \qp \pm \qp\, \dd \qp \big) \wedge \dd \qp 
$$
and, since the integral $\int_h \qtr$ is non-degenerate, equality in \eqref{dis} is equivalent to one of the two equations in \eqref{sde} to be satisfied. 
\end{proof}
A further look at \eqref{ATN} and \eqref{ATP} also shows that 
 \begin{align}\label{sdeBis}
\star \qp\, \dd \qp = - \, \qp\, \dd \qp \, \qquad & \iff \qquad \qp \, \dol \qp = 0   , \nn \\[4pt]
\mbox{and} \qquad 
\star \qp\, \dd \qp = + \, \qp\, \dd \qp \, \qquad & \iff \qquad  \qp \, \dolb \qp = 0    .
 \end{align}
And, it is worth noticing that from the equality \eqref{T} it follows that:
\begin{align}\label{sign}
\int_h \qtr(\qp\, \dd \qp \wedge \dd \qp) =
\int_h \qtr \qp\, \dolb \qp \wedge \dol \qp 
& \leq 0 \quad \mbox{if} \quad \qp\, \dol \qp = 0  \nn \\[4pt]
\int_h \qtr(\qp\, \dd \qp \wedge \dd \qp) =
\int_h \qtr \qp\, \dol \qp \wedge \dolb \qp 
& \geq 0 \quad \mbox{if} \quad \qp\, \dolb \qp = 0   .
\end{align}

\subsection{The twisted action functional}
The quantity on the left hand side of \eqref{dis} is positive and qualifies 
as the `energy' of the projection $\qp$. We have then the following: 
\begin{defi}
On the configuration space made of all projections $\qp \in \Mat_{N}(\Apq)$ 
we consider the $\IR^+$-valued action-functional
\begin{align}\label{actfun}
S_\sigma[\qp] & := \frac{1}{2} \int_h \qtr \, \qp \, (\star \dd \qp) \wedge \dd \qp   .
\end{align}
\end{defi}

We shall not write explicitly the rather complicate equations for the critical points. 
On the other hand, from the consideration above, a class of critical points is obtained when 
the projections $\qp$ satisfies one of the two equations in \eqref{sde} or \eqref{sdeBis}. 
Solutions of these \emph{self-dual/anti-self-dual} equations will be named (\emph{twisted}) 
\emph{sigma-model instantons} or \emph{anti-instantons}, respectively. When this is the case, from the relation \eqref{dis}, the action functional equals the absolute value of a `topological' quantity, given by
\beq\label{top}
\mbox{Top}_\sigma(\qp) = \frac{1}{2} \int_h \qtr(\qp\, \dd \qp \wedge \dd \qp)   ,
\eeq
in parallel with the case for noncommutative tori and Moyal plane of \cite{DKL00,DKL03} 
and \cite{DLL15}.

We shall end this part by giving details for explicit expression to be used later on when looking for solutions. First, a direct computation shows that 
\begin{align}\label{ddb}
\dol \qp \wedge \dolb \qp &= - (X_+ \lt \qp) (X_- \lt \qp) \, \omega_- \wedge \omega_{+} \nn \\[4pt]
& = q^{-2} (X_+ \lt \qp) (X_+ \lt \qp)^* \, \omega_- \wedge \omega_{+} =
q^{2} (X_- \lt \qp)^* (X_- \lt \qp) \, \omega_- \wedge \omega_{+}
\end{align} using \eqref{xstar}, and similarly  
\begin{align}\label{dbd}
\dolb \qp \wedge \dol \qp &= q^{2} (X_- \lt \qp) (X_+ \lt \qp) \, \omega_- \wedge \omega_{+} 
\nn \\[4pt]
& = - (X_+ \lt \qp)^* (X_+ \lt \qp) \, \omega_- \wedge \omega_{+} =
- q^{4} (X_- \lt \qp) (X_- \lt \qp)^* \, \omega_- \wedge \omega_{+}   .
\end{align}  
Then, using both \eqref{ddb} and \eqref{dbd}, one computes
\begin{align*}
 \qp\, \dd \qp \wedge \dd \qp 
& = \qp\, ( \dol \qp \wedge \dolb \qp + \dolb \qp \wedge \dol \qp) \nn \\[4pt]
& = \qp\, \left[  q^{-2} \, (X_+ \lt \qp) (X_+ \lt \qp)^* - (X_+ \lt \qp)^* (X_+ \lt \qp) \right] \, \omega_- \wedge \omega_{+}   ,
\end{align*}
as well as, using also \eqref{hod1}, 
\begin{align*}
\qp\, (\star \dd \qp) \wedge \dd \qp  
& = \qp\, ( \dol \qp \wedge \dolb \qp - \dolb \qp \wedge \dol \qp) \nn \\[4pt]
& = \qp\, \left[  q^{-2} \, (X_+ \lt \qp) (X_+ \lt \qp)^* + (X_+ \lt \qp)^* (X_+ \lt \qp) \right] \, \omega_- \wedge \omega_{+}   .
\end{align*}
All these computations yield for the action-functional:
\begin{align}\label{actfun1}
S_\sigma[\qp]
& = \frac{1}{2} \int_h \qtr \, \qp \left[ q^{-2} \qp \, (X_+ \lt \qp) (X_+ \lt \qp)^* + (X_+ \lt \qp)^* (X_+ \lt \qp) \right] \, \omega_- \wedge \omega_{+} \nn \\[4pt]
& = \frac{1}{2} \int_h \qtr \, \qp \left[ q^{-2} \qp \, (X_+ \lt \qp) (X_+ \lt \qp)^* + 
q^{4} (X_- \lt \qp) (X_- \lt \qp)^* \right] \, \omega_- \wedge \omega_{+}   .
\end{align}
Whereas for the topological term:
 \begin{align}\label{top1}
\mbox{Top}_\sigma(\qp) & = \frac{1}{2} \int_h \qp\, \left[  q^{-2} \, (X_+ \lt \qp) (X_+ \lt \qp)^* - (X_+ \lt \qp)^* (X_+ \lt \qp) \right] \, \omega_- \wedge \omega_{+} \nn \\[4pt]
& = \frac{1}{2} \int_h \qtr \, \qp \left[ q^{-2} \qp \, (X_+ \lt \qp) (X_+ \lt \qp)^* - 
q^{4} (X_- \lt \qp) (X_- \lt \qp)^* \right] \, \omega_- \wedge \omega_{+}   .
\end{align}

\subsection{Twisted sigma-models}
Much of what is written in the previous section does not depend on the details of the geometry of the quantum projective line. We shall try and extracts the minimal requirements for a generalisation. 

One starts with an algebra $\ca$ endowed with a complex structure, that is a breaking of a two-dimensional differential calculus as in Proposition~\ref{2dsph}: 
$$
\Omega^{\bullet}(\ca) = \ca \, \oplus \, \left(\Omega^{1,0}(\ca) \oplus\Omega^{0,1}(\ca) \right) 
\, \oplus \, \ca \, \beta \ ,
$$
with a top central two-form $\beta$ so that $\Omega^{1,1}(\ca)\simeq \ca$.  
The differential is decomposed as $\dd = \dol + \dolb$ with $\dol^2 = \dolb^2 = 
\dol \circ \dolb + \dolb \circ \dol = 0$ (so that $\dd^2=0$), 
with holomorphic forms 
$\Omega^{1,0}(\ca) \simeq\dol(\ca)$ and anti-holomorphic forms
$\Omega^{0,1}(\ca) \simeq\dolb(\ca)$.  

There is a Hodge operator on one-forms  
defined by requiring that 
$$
\star (\dol f) = \dol f  \quad \mbox{and} \quad 
\star (\dolb f) = - \dolb f 
$$
for all $f\in \ca$ 
and extended as a bimodule map. 
Thus $\star$ has values $\pm\, 1$ on holomorphic or anti-holomorphic
one-forms respectively (squaring to the identity). 

Next, there is a normalized positive state $h : \ca \to \IC$, which is twisted, 
that is   
$$ 
h(x y)= h(\sigma (y)x) , \qquad \textup{for} \quad x, y \in \ca   ,
$$ 
for a modular automorphism $\sigma \in \Aut(\ca)$ and such that one can construct a 
non-trivial twisted cyclic $2$-cocycle $\tau$ on $\ca$. This is obtained via a linear functional 
$$
\int_h  : \;\; \Omega^2(\ca) \to \IC, \qquad \int_h a \, \beta := h(a), 
$$
by considering
$$
\tau(a_0,a_1,a_2):= \frac{1}{2} \int_h a_0\, \dd a_1 \wedge \dd a_2   .
$$
This twisted cyclic cocycle $\tau$ is required to be Hochschild-cohomologous to a 
positive twisted Hochschild $2$-cocycle $\varphi$ on $\ca$ defined, by using the complex structure, as
$$
\varphi (a_0, a_1, a_2)=\int_h a_0\, \dol a_1 \, \dolb a_2   .
$$
Finally, one needs a a quantum trace, that is a map $\qtr: \Mat_{N}(\ca) \to \IC$ which is twisted 
by the automorphism $\sigma$, that is 
$$
\qtr(M_1 M_2 ) = \qtr\left( \sigma(M_2) M_1 \right)   .
$$

With all this ingredients, for a projection all projections $\qp \in \Mat_{N}(\ca)$, 
both action-functional $S_\sigma[\qp]$ in \eqref{actfun} and the topological term $\mbox{Top}_\sigma(\qp)$ in \eqref{top} make sense.  Also, both Propositions \ref{apeq} and \ref{bound} are valid and one still has self-duality equations as in \eqref{sde}. 

\section{Connections and holomorphic structures}\label{se:chs}
Before we analyse solutions of the self-duality equations \eqref{sde} or \eqref{sdeBis} we need to recall additional ingredients of the geometry of the quantum-projective line. 

Firstly, the calculus $\Omega^\bullet(\pq)$ is lifted to a 
connection on the quantum principal bundle over $\pq$. 
It leads to
a natural holomorphic structure in the sense that the corresponding covariant derivative is decomposed
into a holomorphic and an anti-holomorphic part. 

\subsection{Connections on equivariant line bundles}\label{se:con}
A connection on the principal bundle over the quantum projective line $\pq$, is a 
covariant splitting of the module of one-forms $\Omega^1(\SU)=\Omega^1_{\mathrm{ver}}(\SU) \oplus\Omega^1_{\mathrm{hor}}(\SU)$. Equivalently, it can be given by a covariant left module 
projection $\Pi : \Omega^1(\SU) \to \Omega^1_{\mathrm{ver}}(\SU)$, i.e.
$\Pi^2=\Pi$ and $\Pi(x \,\alpha) = x\, \Pi(\alpha)$ for
$\alpha\in\Omega^1(\SU)$ and $x \in\mathcal{A}(\SU)$. 
Covariance is the requirement that 
$$
\alpha_{R}^{(1)}\circ \Pi= \Pi \circ \alpha_{R}^{(1)} \ , 
$$ 
with $\alpha_{R}^{(1)}$ the extension to one-forms of the action
$\alpha_{R}$ in \eqref{rco} of the structure group
$\U(1)$. With the left-covariant
three-dimensional calculus on $\ASU$, a basis for
$\Omega^1_{\mathrm{hor}}(\SU)$ is given by
$\omega_{-}$ and $\omega_{+}$. Also, one has
$$
\alpha_{R}^{(1)}(\omega_{z}) = \omega_{z}  \ , \qquad
\alpha_{R}^{(1)}(\omega_{-}) = \omega_{-}\ z^{*}\,^{2} \qquad
\mbox{and} \qquad
\alpha_{R}^{(1)}(\omega_{+}) = \omega_{+}\ z^{2} \ ,
$$
and so a natural choice of connection $\Pi=\Pi_z$ is to define
$\omega_{z}$ to be vertical~\cite{BM93,maj05}, whence
$$
\Pi_{z}(\omega_{z}):=\omega_{z} \quad \mbox{and} \quad
\Pi_{z}(\omega_{\pm}):=0 \ . 
$$
For the associated covariant derivative on right $\Apq$-modules of equivariant elements,  
$$
\nabla:= (\id - \Pi_{z}) \circ \dd\,:\,
\ce ~\longrightarrow~ \ce\otimes_{\Apq} \Omega^1(\pq)  \ , 
$$
one readily proves the Leibniz rule $\nabla(\varphi\cdot f)=
(\nabla\varphi)\cdot f+\varphi\otimes \dd f$ for $\varphi\in\ce$
and $f\in\Apq$.
With the left-covariant two-dimensional calculus on $\Apq$ given in \S\ref{se:cals2} 
(and coming from the left-covariant three-dimensional calculus on $\ASU$ ), we have
then
\beq\label{coder2d}
\nabla \varphi= \left(X_{+}\triangleright\varphi\right)\,\omega_{+}
+\left(X_{-}\triangleright\varphi\right)\,\omega_{-}   . 
\eeq
Given the connection, one can work out an explicit expression for its
curvature, defined to be the $\Apq$-linear (by construction) map 
$$ 
\nabla^2:=\nabla\circ\nabla \, : \, \cl_n ~ \longrightarrow ~ \cl_n 
\otimes_{\Apq}\Omega^2(\pq) \ .
$$
Using \eqref{coder2d}, \eqref{commc3} and the fact that
$\dd\omega_\pm=0$ on $\pq$, by the Leibniz rule one has
\begin{align*}
\nabla \big(\nabla \varphi \big)
&=(X_-\,X_+\lt\varphi)\,\omega_-\wedge\omega_++
(X_+\,X_-\lt\varphi)\,\omega_+\wedge\omega_- \\[4pt]
&= \big((X_-\,X_+-q^{2}\,X_+\,X_-)\lt\varphi\big)\,\omega_-\wedge\omega_+ = \left(X_z \lt \varphi\right)\,\omega_-\wedge\omega_+  , 
\end{align*}
the last quality coming from the relation $X_-\,X_+-q^{2}\,X_+\,X_-=X_z$. 
As expected, the curvature is $\Apq$-linear, as it also follows by observing that $X_z\lt\Apq=0$.

Things become a bit more transparent when taking $\ce=\cl_n$. Then
$X_\pm\lt\varphi\in\cl_{n\pm2}$ for $\varphi\in\cl_n$ and from \eqref{masu} one may conclude that, as required,
$$
\nabla\varphi \ \in \ \cl_{n-2}   \, \omega_- \oplus \cl_{n+2}   \,
\omega_+ \simeq  \cl_{n} \otimes_{\Apq}  \Omega^1(\pq)   .
$$
We can also derive an explicit expression for the corresponding connection one-form 
$\asf_n$ defined by $\varphi\asf_n = \nabla \varphi - \dd
\varphi$ for $\varphi\in\cl_n$. With $X_z$ the vertical vector field
in \eqref{vvf}, using \eqref{exts3} and \eqref{coder2d} one finds
$\varphi\asf_n =  -\left(X_z \lt \varphi\right) \,\omega_z =  q^{n+1}\, [n] \,\varphi\, \omega_z$,
or 
$$
\asf_n=q^{n+1}\, [n] \, \omega_z \ . 
$$
As usual, $\asf_n$ is \emph{not} defined on $\pq$ but rather on the
total space $\SU$ of the bundle, i.e. $ \asf_n \in
\mathrm{Hom}_{\Apq}(\cl_n,\cl_n\otimes_{\Apq} \Omega^1(\SU) )$. 
In terms of it, as a direct consequence of the first identity in \eqref{dformc3}, 
the curvature two-form is then given by
$$
\fsf_n:=\nabla^2 = \dd \asf_n = - q^{n+1}\, [n] \, \omega_{-}\wedge\omega_{+}   ,
$$
using the first relation in \eqref{dformc3}; 
and now $\fsf_n \in \mathrm{Hom}_{\Apq}(\cl_n,\cl_n\otimes_{\Apq} \Omega^1(\pq))$. 

\subsection{Holomorphic structures}\label{se:hol}
The connection in \eqref{coder2d} lends itself to be naturally decomposed
into a holomorphic and an anti-holomorphic part, $\nabla =
\nabla^{\dol} +\nabla^{\dolb}$. These are:  
\beq\label{coher}
\nabla^{\dol} \varphi  =  \left(X_{+}\lt \varphi\right)\,\omega_{+}
\quad \mbox{and} \quad
\nabla^{\dolb} \varphi  =
\left(X_{-}\triangleright\varphi\right)\,\omega_{-}   ,
\eeq
with the corresponding Leibniz rules: for all $\varphi\in\ce$ and $f\in\Apq$ it holds that, 
$$
\nabla^{\dol}(\varphi\cdot f) = \big(\nabla^{\dol}\varphi\big)\cdot f
+\varphi\otimes\dol f \quad \mbox{and} \quad
\nabla^{\dolb}(\varphi\cdot f) = \big(\nabla^{\dolb}\varphi\big)\cdot
f+ \varphi\otimes \dolb f \ .
$$

They are both flat, that is,
$
(\nabla^{\dol})^2=0=(\nabla^{\dolb}\,)^2   .
$ 
Holomorphic `sections' are elements $\varphi\in\ce$ which satisfy
$
\nabla^{\dolb} \varphi = 0 \  
$
and we denotes the collection of them as $H^0(\ce, \nabla^{\dolb})$. 

For $\ce=\cl_n$, results on holomorphic sections were obtained in~\cite{KLvS}.
{}From the actions given in \eqref{lact} we see that $F\lt a^{s} =0$ and
$F\lt c^{s} =0$ for any integer $s\in\IN$, while $F\lt a^{*}\,^{s} \not= 0$
and $F\lt c^{*}\,^{s} \not= 0$ for any $s\in\IN$. Then, from the
expressions \eqref{eqmap} for generic equivariant elements, we see that
there are no holomorphic elements in $\cl_n$ for $n>0$. On the other
hand, for $n\leq 0$ the elements $a^{\mn-\mu}\, c^{\mu}$ for $\mu = 0,1, \dots,\mn$,  are holomorphic:
\beq\label{holn}
\nabla^{\dolb} \big( a^{\mn-\mu}\, c^{\mu} \big) = 0 \ . 
\eeq
Since $\ker \dolb= \IC$ (as only the constant functions on $\pq$ do
not contain the generator $a^*$ or $c^*$), so that the only
holomorphic functions on $\pq$ are the constants \cite{KLvS}, these are the only
invariants in degree $n$. We may conclude that holomorphic equivariant
elements are all polynomials in two variables $a,c$ with the
commutation relation $a\,c = q\,c\,a$, which defines the coordinate
algebra of the quantum plane. We have the following:
\begin{theo}\label{holsec}
Let $n$ be a positive integer. Then
$$
H^0(\cl_n, \nabla^{\dolb}) = 0,  \quad \mbox{while} \quad H^0(\cl_{-n}, \nabla^{\dolb}) 
\simeq \IC^{n+1}   .
$$ 
Moreover the space $R:=\bigoplus_{n\geq 0} H^0(\cl_{-n}, \nabla^{\dolb})$ has a ring structure and is isomorphic to 
$$
R \simeq \IC \langle a,c \rangle /  (ac - q ca)   .
$$
\end{theo}

Thus there is a quantum homogeneous coordinate ring for $\pq$ that results into the 
coordinate ring of the quantum plane for the variable $a$ and $c$ with relation $ac - q ca=0$.  

In a similar fashion one shows that:
\beq\label{aholn}
\nabla^{\dol} \big( a^{*}\,^{ \mu}\,c^{*}\,^{ n-\mu} \big) = 0 \ , 
\eeq
so that 
the elements $a^{*}\,^{ \mu}\,c^{*}\,^{ n-\mu}$ for $\mu = 0,1, \dots,n$,  are anti-holomorphic. 

\section{Twisted sigma-model solitons}
In computing the integral of the curvature on bundles over the quantum projective line in \cite{LRZ09} we already come across to the equations \eqref{sdeBis} for a class of solutions, albeit not thinking or interpreting them as noncommutative instantons. 
We shall now work out explicitly solutions of the self-duality equations \eqref{sde} by using results from \cite{KLvS} (on the holomorphic structure on $\pq$ described in Sect.~\ref{se:hol} above), in a way that can be extended to more general quantum projective spaces $\pqn$ in any dimension. 

\subsection{Explicit solutions}
The line bundles in \eqref{libudual} can be realised as right modules: $\cl_n\simeq \qp (\Apq)^{\mn+1}$ with projections $\qp$ written explicitly \cite{BM98,HM98}.
These are elements $\qpn$ and $\qpp$ in $\Mat_{\mn+1}(\Apq)$ (for $n\geq0$ and $n\leq0$ respectively) whose explicit form we give now following \cite{LRZ09} 
(if exchanging the sign of the label $n$ there due to the passage from left-modules there to right ones here). Consider then vector valued functions 
 \begin{align*}
 (\Phi_{(n)})_{\mu} &= \sqrt{\alpha_{n,\mu}} \,\, c^{n-\mu}a^{\mu}   , \quad &\mu=0, 1, \dots, n   , 
\qquad \mbox{for} \quad n \geq 0 \nonumber  \\[4pt]
(\Psi_{(n)})_{\mu} & = \sqrt{\beta_{n,\mu}} \,\, c^{* \mu}a^{* \mn-\mu}   , 
\quad &\mu=0, 1, \dots, \mn   , 
\qquad \mbox{for} \quad n \leq 0   ,
 \end{align*}  
with coefficients (and convention $\prod_{j=0}^{-1}(\cdot)=1$) 
\begin{align*}
\alpha_{n,\mu} =\prod\nolimits_{j=0}^{n-\mu-1}\left(\frac{1-q^{2\left(n-j\right)}}
{1-q^{2\left(j+1\right)}}\right) \quad \mbox{and} \quad 
\beta_{\mn,\mu} =q^{2\mu}\prod\nolimits_{j=0}^{\mu-1}\left(\frac{1-q^{-2\left(\mn-j\right)}}
{1-q^{-2\left(j+1\right)}}\right),
\end{align*}
for $\mu = 0, \dots, n$ or $\mu = 0, \dots, \mn$ respectively. These coefficients are such that
\begin{align}\label{id2}
 \Phi_{(n)}^* \, \Phi_{(n)} &= \sum\nolimits_{\mu=0}^{n} (\Phi_{(n)}^*)_{\mu} (\Phi_{(n)})_{\mu} 
 \nn\\
 & =
\sum\nolimits_{\mu=0}^{\mn} 
 \alpha_{n,\mu}\, a^{* \mu} c^{* n-\mu} c^{n-\mu} a^{\mu} = (a^{*}a+c^{*}c)^{n} = 1,
\end{align}
and analogously
\begin{align}\label{id1}
\Psi_{(n)}^* \Psi_{(n)} 
&= \sum\nolimits_{\mu=0}^{\mn} (\Psi_{(n)}^*)_\mu(\Psi_{(n)})_\mu \nn \\
& = \sum\nolimits_{\mu=0}^{\mn} 
 \beta_{n,\mu}\, a^{\mn-\mu} c^{\mu} c^{* \mu} a^{*\mn-\mu} =
 (aa^{*}+q^{2}cc^{*})^{\mn} = 1 .
\end{align}

\noindent
We have then projections: 
\begin{align}\label{qpron}
\qpn &:= \Phi_{(n)} \Phi_{(n)}^*, \qquad \mbox{for} \quad n\geq 0, \nn \\[4pt]
\qpn_{\mu\nu} &\:= \sqrt{\alpha_{n,\mu}\alpha_{n,\nu}} \,\, c^{n-\mu}a^{\mu}
a^{*\nu}c^{*n-\nu} \in \Apq ,  
\end{align}
\begin{align}\label{qpro}
\mbox{and }\qquad \qpp &:= \Psi_{(n)} \Psi_{(n)}^* , \qquad  \mbox{for} \quad n\leq 0 , \nn \\[4pt]
\qpp_{\mu\nu} &\:=\sqrt{\beta_{n,\mu}\beta_{n,\nu}} \,\, 
c^{*\mu}a^{*\mn-\mu}a^{\mn-\nu}c^{\nu} \in \Apq . \qquad\qquad
\end{align}


\noindent
By construction, they are idempotents: $(\qpn)^2=\qpn$ and $(\qpp)^2=\qpp$ 
(and self-adjoint).

To following result is in the proof of \cite[Lemma~4.3]{LRZ09}. We shall re-derive the result by more direct methods using the holomorphic structure on the bundles $\cl_n$ given in Sect.~\ref{se:hol}.

\begin{lemm}\label{lemproj}
With the two-dimensional calculus on the projective line $\pq$ given in Sect.~\ref{se:cals2}, 
if $\qpn$ is the projection given in \eqref{qpron}, it holds that
\beq\label{antisf}
\qpn \, \dol \qpn = 0   , 
\eeq
while if $\qpp$ is the projection given in \eqref{qpro}, it holds that
\beq\label{sf}
\qpp\,  \dolb \qpp = 0   . 
\eeq
\end{lemm}
\begin{proof}
By \eqref{holn} and \eqref{aholn} (and working component-wise), for any $n\geq 0$ the vector $ \Phi_{(n)}$ is holomorphic, that is $\nabla^{\dolb} \Phi_{(n)} = 0$, while the vector $ \Psi_{(n)}$ is anti-holomorphic, that is $\nabla^{\dol} \Psi_{(n)} = 0$; and by conjugation one has also 
$\nabla^{\dolb} \Psi_{(n)}^* = 0$ and $\nabla^{\dol} \Phi_{(n)}^* = 0$. Using these and the second identity in \eqref{cotb} (the twisted Leibniz rule for $X_+$) one easily computes:
$$
\dol \qpn = (X_+ \lt \qpn) \omega_+ = (X_+ \lt \Phi_{(n)})(K^2 \lt \Phi_{(n)}^*)  \omega_+ = 
q^n (X_+ \lt \Phi_{(n)}) \Phi_{(n)}^* \omega_+
$$
and, since $\Phi_{(n)}^* \Phi_{(n)} = 1$ and $X_+ \lt \Phi_{(n)}^*=0$, 
\begin{align*}
\qpn \, \dol \qpn &= 
q^n \Phi_{(n)} \left[ \Phi_{(n)}^* (X_+ \lt \Phi_{(n)}) \right] \Phi_{(n)}^* \omega_+  \nn\\
&= - q^n \Phi_{(n)} \left[ (X_+ \lt \Phi_{(n)}^*) (K^2 \lt \Phi_{(n)}) \right] \Phi_{(n)}^* \omega_+ = 0   ,
\end{align*}
that is the stated \eqref{antisf}. The second equation \eqref{sf} is shown in a similar manner. 
\end{proof}

\begin{lemm}\label{lemprojbis}
With the two-dimensional calculus on the projective line $\pq$ given in Sect.~\ref{se:cals2}, 
if $\qpn$ is the projection given in \eqref{qpron}, it holds that
\beq\label{pdpdp-}
\qpn \,\dd \qpn \wedge\, \dd \qpn = -q^{n+1} [n] \, \qpn \, \omega_{-}\wedge\omega_{+}   .
\eeq
While if $\qpp$ is the projection given in \eqref{qpro}, it holds that
\beq\label{pdpdp+}
\qpp\,\dd \qpp \wedge\, \dd \qpp\, = q^{ - \mn + 1} [ \mn ] \, \qpp\, \omega_{-}\wedge\omega_{+}   .
\eeq
\end{lemm}
\begin{proof}
We need to compute te quantity $\qpn \, \dolb \qpn$. Just as for Lemma~\ref{lemproj}, one has
$$
\dolb \qpn = (X_- \lt \qpn) \omega_- = \Phi_{(n)} (X_- \lt \Phi_{(n)}^*) \omega_-
$$
and in turn, using the fact that $\Phi_{(n)}^* \Phi_{(n)}=1$,  one gets:
$$
\qpn \, \dolb \qpn = \Phi_{(n)} \Phi_{(n)}^* \Phi_{(n)} (X_- \lt \Phi_{(n)}^*) \omega_- = 
\Phi_{(n)} (X_- \lt \Phi_{(n)}^*) \omega_-   .
$$
Finally, 
\begin{align*}
\qpn \,\dd \qpn \wedge\, \dd \qpn = \qpn \,\dolb \qpn \wedge\, \dol \qpn
& = q^n \Phi_{(n)} (X_- \lt \Phi_{(n)}^*) \omega_- (X_+ \lt \Phi_{(n)}) \Phi_{(n)}^* \omega_+ \\
& = q^{n+2} \Phi_{(n)} \left[ (X_- \lt \Phi_{(n)}^*)(X_+ \lt \Phi_{(n)}) \right] \Phi_{(n)}^* \, 
\omega_- \wedge \omega_+   .
\end{align*}  
For the term in the square bracket, using the twisted Leibniz rule for $X_+$ and $X_-$ and the holomorphic properties of the vector $\Phi_{(n)}$ one computes:
\begin{align*}
0 & = X_+ \lt 1 = X_+ \lt ( \Phi_{(n)}^* \, \Phi_{(n)}) = \Phi_{(n)}^* (X_+ \lt \Phi_{(n)}) \\[4pt]
\Rightarrow \quad 0 & = X_- \lt X_+ \lt 1 = X_- \lt (\Phi_{(n)}^* (X_+ \lt \Phi_{(n)}))  \\
& = \Phi_{(n)}^* X_- \lt (X_+ \lt \Phi_{(n)}) + (X_- \lt \Phi_{(n)}^*) (K^2 \lt X_+ \lt \Phi_{(n)}) \\
& = q^{1-n} \big( [n] + q \, (X_- \lt \Phi_{(n)}^*) (X_+ \lt \Phi_{(n)}) \big)
\end{align*}
since $K^2 \lt X_+ \lt \Phi_{(n)} = q^{2-n} X_+ \lt \Phi_{(n)}$ and 
$X_- \lt (X_+ \lt \Phi_{(n)}) = q^{1-n} [n]\Phi_{(n)}$ by a direct if lengthy computation. By substituting in the square bracket above one gets the equality \eqref{pdpdp-}. The result 
 \eqref{pdpdp+} is shown along similar lines. 
\end{proof}

Finally, it holds that \cite[Lemma~5.1]{LRZ09} that (recall the change in notation $n\to-n$) 
$$
\qtr \qpn\ = q^{-n}  \quad \mbox{and} \quad \qtr \qpp\ = q^{\mn}   .
$$
This, together with \eqref{pdpdp-} and \eqref{pdpdp-} then leads to the computations:
\begin{align*}
\mbox{Top}_\sigma(\qpn) &= \frac{1}{2} \int_h \qtr(\qpn\, \dd \qpn \wedge \qpn \qp) 
= -q [n] \, \int_h \, \omega_{-}\wedge\omega_{+} = -q \, [n]   ,  
\nn \\[4pt]
\mbox{and} \qquad 
\mbox{Top}_\sigma(\qpp) &= \frac{1}{2} \int_h \qtr(\qpp\, \dd \qpp \wedge \qpp \qp) 
= q [\mn] \, \int_h \, \omega_{-}\wedge\omega_{+} = q \, [\mn]   .
\end{align*}
Notice the agreement with the sign rule in \eqref{sign} for self-dual or anti-self-dual solutions in \eqref{antisf} or \eqref{sf}. 
As mentioned, in \cite{NT05,wa07} these were given as the $q$-index of the Dirac operator on the quantum projective line $\pq$, thus giving them a `topological' meaning. 

The results of the present paper ought to be directly generalizable to the quantum projective spaces $\pqn$, 
using the complex and holomorphic structures on $\pqn$ given in \cite{KM11,KM11b} 
that extends those introduced in \cite{KLvS}. Higher dimensional twisted sigma-model (anti)-instantons are to be found among the projections on $\pqn$ constructed in \cite{DL10,DD10}.


\end{document}